\documentclass[12pt,reqno]{amsart}
\usepackage{eurosym}
\usepackage{amsfonts}
\usepackage{amsmath}
\usepackage{amssymb,amsmath}
\usepackage{amscd}
\usepackage{epsfig}
\usepackage{graphicx}
\usepackage[bookmarksnumbered, colorlinks, plainpages]{hyperref}

\hypersetup{colorlinks=true,linkcolor=red, anchorcolor=green, citecolor=cyan, urlcolor=red, filecolor=magenta, pdftoolbar=true}
\newtheorem{theorem}{Theorem}[section]

\newtheorem{proposition}[theorem]{Proposition}
\newtheorem{corollary}[theorem]{Corollary}
\theoremstyle{definition}
\newtheorem{definition}[theorem]{Definition}
\newtheorem{example}[theorem]{Example}

\newtheorem{question}[theorem]{Question}
\newtheorem{notation}[theorem]{Notation}

\theoremstyle{remark}
\newtheorem{remark}[theorem]{Remark}
\numberwithin{equation}{section}

\begin{document}
\title[Fixed-Circle Results]{New fixed-circle results related to $F_{c}$%
-contractive and $F_{c}$-expanding mappings on metric spaces}
\author[N. Mlaiki\MakeLowercase{,} N. \"{O}zg\"{u}r \MakeLowercase{and} N. Ta%
\c{s}]{Nabil MLAIKI$^{1}$\MakeLowercase{,} N\.{I}HAL \"{O}ZG\"{U}R$^{2}$ %
\MakeLowercase{and} N\.{I}HAL TA\c{S}$^{2}$}
\address{$^{1}$Department of Mathematics and General Sciences, Prince Sultan
University, Riyadh, Saudi Arabia.}
\email{\textcolor[rgb]{0.00,0.00,0.84}{nmlaiki@psu.edu.sa; nmlaiki2012@gmail.com}}
\address{$^{2}$Bal\i kesir University, Department of Mathematics, 10145
Bal\i kesir, TURKEY.}
\email{%
\textcolor[rgb]{0.00,0.00,0.84}{nihal@balikesir.edu.tr;
nihaltas@balikesir.edu.tr}}

\begin{abstract}
The fixed-circle problem is a recent problem about the study of geometric
properties of the fixed point set of a self-mapping on metric (resp.
generalized metric) spaces. The fixed-disc problem occurs as a natural
consequence of this problem. Our aim in this paper, is to investigate new
classes of self-mappings which satisfy new specific type of contraction on a
metric space. We see that the fixed point set of any member of these classes
contains a circle (or a disc) called the fixed circle (resp. fixed disc) of
the corresponding self-mapping. For this purpose, we introduce the notions
of an $F_{c}$-contractive mapping and an $F_{c}$-expanding mapping.
Activation functions with fixed circles (resp. fixed discs) are often seen
in the study of neural networks. This shows the effectiveness of our
fixed-circle (resp. fixed-disc) results. In this context, our theoretical
results contribute to future studies on neural networks.
\end{abstract}

\maketitle


\setcounter{page}{1}


\section{\textbf{Introduction}}

\label{sec:intro} In the last few decades, the Banach contraction principle
has been generalized and studied by different approaches such as to
generalize the used contractive condition (see \cite{Altun}, \cite{Caristi},
\cite{Chatterjea}, \cite{Ciric}, \cite{Djafari}, \cite{Edelstein}, \cite%
{Kannan}, \cite{Nemytskii}, \cite{Olgun} and \cite{Rhoades} for more
details) and to generalize the used metric space (see \cite{an}, \cite%
{b-metric}, \cite{2-metric}, \cite{Karayilan}, \cite{G-metric}, \cite%
{Mitrovic}, \cite{Pasic}, \cite{Sedghi-S-metric} and \cite{Sihag} for more
details). Recently, some fixed-circle theorems have been introduced as a
geometrical direction of generalization of the fixed-point theorems (see
\cite{Ozgur-malaysian}, \cite{Ozgur-fixed-circle-S-metric}, \cite%
{Ozgur-circle-tez} and \cite{Ozgur-Aip} for more details).

Let $(X,d)$ be a metric space and $f$ be a self-mapping on $X$. First, we
recall that the circle $C_{u_{0},\rho }=\left\{ u\in X:d(u,u_{0})=\rho
\right\} $ is a fixed circle of $f$ if $fu=u$ for all $u\in C_{u_{0},\rho }$
(see \cite{Ozgur-malaysian}). Similarly, the disc $D_{u_{0},\rho }=\left\{
u\in X:d(u,u_{0})\leq \rho \right\} $ is called a fixed disc of $f$ if $fu=u$
for all $u\in D_{u_{0},\rho }$. There are some examples of self-mappings
such that the fixed point set of the self-mapping contains a circle (or a
disc). For example, let us consider the metric space $\left(
\mathbb{C}
,d\right) $ with the metric
\begin{equation}
d\left( z_{1},z_{2}\right) =\left\vert x_{1}-x_{2}\right\vert +\left\vert
y_{1}-y_{2}\right\vert +\left\vert x_{1}-x_{2}+y_{1}-y_{2}\right\vert ,
\label{first metric}
\end{equation}%
defined for the complex numbers $z_{1}=x_{1}+iy_{1}$ and $z_{2}=x_{2}+iy_{2}$%
. We note that the metric defined in (\ref{first metric}) is the metric
induced by the norm function%
\begin{equation*}
\left\Vert z\right\Vert =\left\Vert x+iy\right\Vert =\left\vert x\right\vert
+\left\vert y\right\vert +\left\vert x+y\right\vert \text{,}
\end{equation*}%
(see Example 2.4 in \cite{Ozgur norm}). The circle $C_{0,1}$ is seen in the
following figure which is drawn using Mathematica \cite{Mathematica}. Define
the self-mapping $f_{1}$ on $\mathbb{C}$ as follows$:$
\begin{equation*}
f_{1}z=\left\{
\begin{array}{ccc}
z & ; & x\leq 0,y\geq 0\text{ or }x\geq 0,y\leq 0 \\
-y+\frac{1}{2}+i\left( -x+\frac{1}{2}\right)  & ; & x>0,y>0 \\
-y-\frac{1}{2}+i\left( -x-\frac{1}{2}\right)  & ; & x<0,y<0%
\end{array}%
\right. \text{,}
\end{equation*}%
for each $z=x+iy\in \mathbb{C}$, then clearly, the fixed point set of $f_{1}$
contains the circle $C_{0,1}$, that is, $C_{0,1}$ is a fixed circle of $f_{1}
$. Therefore, the study of geometric properties of the fixed point set of a
self-mapping seems to be an interesting problem in case where the fixed
point is non unique.

\begin{figure}[h]
\centering
\includegraphics[width=6.5cm]{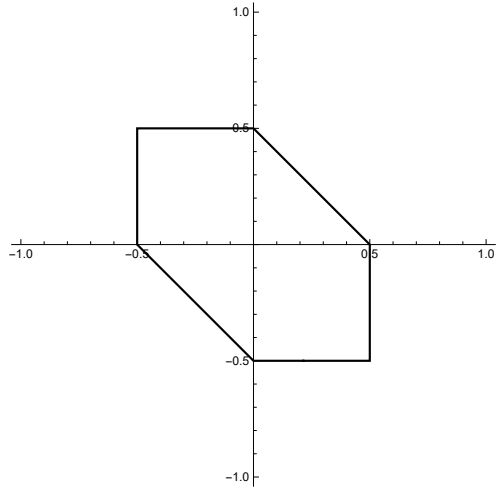}
\caption{\Small The graph of the circle $C_{0,1}$.}
\label{fig:1}
\end{figure}

On the other hand, fixed points of self-mappings play an important role in
the study of neural networks. For example, in \cite{Mandic}, it was pointed
out that fixed points of a neural network can be determined by fixed points
of the employed activation function. If the global input-output relationship
in a neural network can be considered in the framework of M\"{o}bius
transformations, then the existence of one or two fixed points of the neural
network is guaranteed (see \cite{Jones and Singerman} for basic algebraic
and geometric properties of M\"{o}bius transformations). Some possible
applications of theoretical fixed-circle results to neural networks have
been investigated in the recent studies \cite{Ozgur-malaysian} and \cite%
{Ozgur-Aip}.

Next, we remind the reader of the following theorems on a fixed circle.

\begin{theorem}
\cite{Ozgur-malaysian} \label{thm11} Let $(X,d)$ be a metric space and
consider the map%
\begin{equation}
\varphi :X\rightarrow \left[ 0,\infty \right) \text{, }\varphi (u)=d(u,u_{0})%
\text{,}  \label{funtion_phi}
\end{equation}%
for all $u\in X$. If there exists a self-mapping $f:X\rightarrow X$
satisfying

$(C1)$ $d(u,fu)\leq \varphi (u)-\varphi (fu)$

and

$(C2)$ $d(fu,u_{0})\geq \rho $, \newline
for each $u\in C_{u_{0},\rho }$, then the circle $C_{u_{0},\rho }$ is a
fixed circle of $f$.
\end{theorem}

\begin{theorem}
\cite{Ozgur-malaysian} \label{thm12} Let $(X,d)$ be a metric space and
consider the map $\varphi $ as defined in $($\ref{funtion_phi}$)$. Also,
assume that $f:X\rightarrow X$ satisfies the following conditions:

$(C1)^{\ast }$ $d(u,fu)\leq \varphi (u)+\varphi (fu)-2\rho $

and

$(C2)^{\ast }$ $d(fu,u_{0})\leq \rho $, \newline
for each $u\in C_{u_{0},\rho }$, then the circle $C_{u_{0},\rho }$ is a
fixed circle of $f$.
\end{theorem}

\begin{theorem}
\cite{Ozgur-malaysian} \label{thm13} Let $(X,d)$ be a metric space and
consider the map $\varphi $ as defined in $($\ref{funtion_phi}$)$. Also,
assume that $f:X\rightarrow X$ satisfies the following conditions:

$(C1)^{\ast \ast }$ $d(u,fu)\leq \varphi (u)-\varphi (fu)$

and

$(C2)^{\ast \ast }$ $hd(u,fu)+d(fu,u_{0})\geq \rho $, \newline
for each $u\in C_{u_{0},\rho }$ and some $h\in \left[ 0,1\right) $, then the
circle $C_{u_{0},\rho }$ is a fixed circle of $f$.
\end{theorem}

\begin{theorem}
\cite{Ozgur-Aip} \label{aip} Let $(X,d)$ be a metric space and assume that
the mapping $\varphi _{\rho }:%
\mathbb{R}
^{+}\cup \left\{ 0\right\} \rightarrow
\mathbb{R}
$ be defined by%
\begin{equation}
\varphi _{\rho }(x)=\left\{
\begin{array}{ccc}
x-\rho & ; & x>0 \\
0 & ; & x=0%
\end{array}%
\right. \text{,}  \label{phi_mapping}
\end{equation}%
for all $x\in
\mathbb{R}
^{+}\cup \left\{ 0\right\} $. If there exists a self-mapping $f:X\rightarrow
X$ satisfying

\begin{enumerate}
\item $d(fu,u_{0})=\rho $ for each $u\in C_{u_{0},\rho }$,

\item $d(fu,fv)>\rho $ for each $u,v\in C_{u_{0},\rho }$ and $u\neq v$ ,

\item $d(fu,fv)\leq d(u,v)-\varphi _{\rho }(d(u,fu))$ for each $u,v\in
C_{u_{0},\rho }$,
\end{enumerate}

then the circle $C_{u_{0},\rho }$ is a fixed circle of $f$.
\end{theorem}

This manuscript is structured as follows; in Section \ref{sec:1}, we give
some generalizations of Theorems \ref{thm11}, \ref{thm12} and \ref{thm13}.
In Section \ref{sec:2}, we present the definitions of an \textquotedblleft $%
F_{c}$-contraction\textquotedblright\ and an \textquotedblleft $F_{c}$%
-expanding map\textquotedblright\ where we prove new theorems on a fixed
circle. In section \ref{sec:3}, we consider the fixed point sets of some
activation functions frequently used in the study of neural networks with a
geometric viewpoint. This shows the effectiveness of our fixed-circle
results. In section \ref{sec:4}, we present some open problems for future
works. Our results show the importance of the geometry of fixed points of a
self-mapping when the fixed point is non unique.

\section{\textbf{New Fixed-Circle Theorems for Some Contractive Mappings}}

\label{sec:1} First, we give a fixed-circle theorem using an auxiliary
function.

\begin{theorem}
\label{thm21} Let $(X,d)$ be a metric space, $f$ be a self-mapping on $X$
and the mapping $\theta _{\rho }:%
\mathbb{R}
\rightarrow
\mathbb{R}
$ be defined by%
\begin{equation*}
\theta _{\rho }(x)=\left\{
\begin{array}{ccc}
\rho & ; & x=\rho \\
x+\rho & ; & x\neq \rho%
\end{array}%
\right. \text{,}
\end{equation*}%
for all $x\in
\mathbb{R}
$ and $\rho \geq 0$. Suppose that

\begin{enumerate}
\item $d(fu,u_{0})\leq \theta _{\rho }(d(u,u_{0}))+Ld(u,fu)$ for some $L\in
\left( -\infty ,0\right] $ and each $u\in X$,

\item $\rho \leq d(fu,u_{0})$ for each $u\in C_{u_{0},\rho }$,

\item $d(fu,fv)\geq 2\rho $ for each $u,v\in C_{u_{0},\rho }$ and $u\neq v$,

\item $d(fu,fv)<\rho +d(v,fu)$ for each $u,v\in C_{u_{0},\rho }$ and $u\neq
v $,
\end{enumerate}

then $f$ fixes the circle $C_{u_{0},\rho }$.
\end{theorem}

\begin{proof}
Let $u\in C_{u_{0},\rho }$ be an arbitrary point. By the conditions (1) and
(2), we have%
\begin{equation*}
d(fu,u_{0})\leq \theta _{\rho }(d(u,u_{0}))+Ld(u,fu)=\rho +Ld(u,fu)
\end{equation*}%
and so%
\begin{equation}
\rho \leq d(fu,u_{0})\leq \rho +Ld(u,fu)\text{. }  \label{eqn21}
\end{equation}%
We have two cases.

\textbf{Case 1.} Let $L=0$. Then we find $d(fu,u_{0})=\rho $ by (\ref{eqn21}%
), that is, we have $fu\in C_{u_{0},\rho }$. Then $d(u,fu)=0$ or $%
d(u,fu)\neq 0$. Assume $d(u,fu)\neq 0$ for $u\in C_{u_{0},\rho }$. Since $%
u\neq fu$, from the condition (3), we obtain%
\begin{equation}
d(fu,f^{2}u)\geq 2\rho \text{.}  \label{eqn22}
\end{equation}%
Also using the condition (4), we get%
\begin{equation*}
d(fu,f^{2}u)<\rho +d(fu,fu)
\end{equation*}%
and hence%
\begin{equation*}
d(fu,f^{2}u)<\rho \text{.}
\end{equation*}%
which contradicts the inequality (\ref{eqn22}). Therefore, it should be $%
d(u,fu)=0$ which implies $fu=u$.

\textbf{Case 2.} Let $L\in \left( -\infty ,0\right) $. If $d(u,fu)\neq 0$ we
get a contradiction by (\ref{eqn21}). Hence it should be $d(u,fu)=0$.

Thereby, we obtain $fu=u$ for all $u\in C_{u_{0},\rho }$, that is, $%
C_{u_{0},\rho }$ is a fixed circle of $f$. In other words, the fixed point
set of $f$ contains the circle $C_{u_{0},\rho }$.
\end{proof}

\begin{remark}
Notice that, if we consider the case $L\in \left( -\infty ,0\right) $ in the
condition $(1)$ of Theorem \ref{thm21} for $u\in C_{u_{0},\rho }$, then we
get%
\begin{equation*}
-Ld(u,fu)\leq \theta _{\rho
}(d(u,u_{0}))-d(fu,u_{0})=d(u,u_{0})-d(fu,u_{0})=\varphi (u)-\varphi (fu)%
\text{ }
\end{equation*}%
and hence%
\begin{equation*}
-Ld(u,fu)\leq \varphi (u)-\varphi (fu)\text{.}
\end{equation*}%
For $L=-1$, we obtain%
\begin{equation*}
d(u,fu)\leq \varphi (u)-\varphi (fu)\text{.}
\end{equation*}%
This means that the condition $(C1)$ $($resp. the condition $(C1)^{\ast \ast
})$ is satisfied for this case.

Clearly, the condition $(2)$ of Theorem \ref{thm21} is the same as condition
$(C2)$. On the other hand, if the condition $(2)$ of Theorem \ref{thm21} is
satisfied then the condition $(C2)^{\ast \ast }$ is satisfied. Consequently,
Theorem \ref{thm21} is a generalization of Theorem \ref{thm11} and Theorem %
\ref{thm13} for the cases $L\in \left( -\infty ,0\right) \setminus \left\{
-1\right\} $. For the case $L=-1$, Theorem \ref{thm21} coincides with
Theorem \ref{thm11} and it is a special case of Theorem \ref{thm13}.
\end{remark}

Next, we present some illustrative examples.

\begin{example}
\label{exm:21} Let $\left(
\mathbb{R}
,d\right) $ be the metric space with the usual metric $d(x_{1},x_{2})=\left%
\vert x_{1}-x_{2}\right\vert $ and consider the circle $C_{0,1}=\left\{
-1,1\right\} $. If we define the self-mapping $f_{1}:%
\mathbb{R}
\rightarrow
\mathbb{R}
$ as%
\begin{equation*}
f_{1}x=\left\{
\begin{array}{ccc}
3x^{2}+x-3 & ; & x\in \left\{ -1,1\right\} \\
0 & ; & \text{otherwise}%
\end{array}%
\right. \text{,}
\end{equation*}%
for each $x\in
\mathbb{R}
$, then it is not difficult to see that $f_{1}$ satisfies the hypothesis of
Theorem \ref{thm21} for the circle $C_{0,1}$ and $L=\frac{-1}{2}$. Clearly, $%
C_{0,1}$ is the fixed circle of $f_{1}$.
\end{example}

\begin{example}
\label{exm:22} Consider $\left(
\mathbb{R}
,d\right) $ to be the usual metric space and the circle $C_{0,2}=\left\{
-2,2\right\} $. Define $f_{2}:%
\mathbb{R}
\rightarrow
\mathbb{R}
$ by%
\begin{equation*}
f_{2}x=\left\{
\begin{array}{ccc}
2 & ; & x=-2 \\
-2 & ; & x=2 \\
0 & ; & \text{otherwise}%
\end{array}%
\right. \text{,}
\end{equation*}%
for each $x\in
\mathbb{R}
$, then $f_{2}$ does not satisfy the condition $(1)$ of Theorem \ref{thm21}
for each $x\in C_{0,2}$ and for any $L\in \left( -\infty ,0\right) $. Also, $%
f_{2}$ does not satisfy the condition $(4)$ for each $x\in C_{0,2}$ and for
any $L\in \left( -\infty ,0\right] $. Clearly, $f_{2}$ does not fix $C_{0,2}$
and this example shows that the condition $(4)$ is crucial in Theorem \ref%
{thm21}.
\end{example}

\begin{example}
\label{exm:23} Consider $\left(
\mathbb{R}
,d\right) $ to be the usual metric space and the circles $C_{0,1}=\left\{
-1,1\right\} $ and $C_{0,2}=\left\{ -2,2\right\} $. If we define $f_{3}:%
\mathbb{R}
\rightarrow
\mathbb{R}
$ as%
\begin{equation*}
f_{3}x=\left\{
\begin{array}{ccc}
x & ; & x\in C_{0,1}\cup C_{0,2} \\
0 & ; & \text{otherwise}%
\end{array}%
\right. \text{,}
\end{equation*}%
for each $x\in
\mathbb{R}
$, then $f_{3}$ satisfies the hypothesis of Theorem \ref{thm21} for each of
the circles $C_{0,1}$ and $C_{0,2}$ and for any $L\in \left[ -1,0\right] $.
Clearly, $C_{0,1}$ and $C_{0,2}$ are the fixed circles of $f_{3}$.
\end{example}

We give another fixed-circle result.

\begin{theorem}
\label{thm22} Let $(X,d)$ be a metric space, $f$ be a self-mapping on $X$
and the mapping $\theta _{\rho }:%
\mathbb{R}
\rightarrow
\mathbb{R}
$ be defined by%
\begin{equation*}
\theta _{\rho }(x)=\left\{
\begin{array}{ccc}
\rho & ; & x=\rho \\
x+\rho & ; & x\neq \rho%
\end{array}%
\right. \text{,}
\end{equation*}%
for all $x\in
\mathbb{R}
$ and $\rho \geq 0$. Suppose that

\begin{enumerate}
\item $2d(u,u_{0})-d(fu,u_{0})\leq \theta _{\rho }(d(u,u_{0}))+Ld(u,fu)$ for
some $L\in \left( -\infty ,0\right] $ and each $u\in X$,

\item $d(fu,u_{0})\leq \rho $ for each $u\in C_{u_{0},\rho }$,

\item $d(fu,fv)\geq 2\rho $ for each $u,v\in C_{u_{0},\rho }$ and $u\neq v$,

\item $d(fu,fv)<\rho +d(v,fu)$ for each $u,v\in C_{u_{0},\rho }$ and $u\neq
v $,

then the self-mapping $f$ fixes the circle $C_{u_{0},\rho }.$
\end{enumerate}
\end{theorem}

\begin{proof}
Consider $u\in C_{u_{0},\rho }$ to be an arbitrary point. Using the
conditions (1) and (2), we get%
\begin{equation*}
2d(u,u_{0})-d(fu,u_{0})\leq d(u,u_{0})+Ld(u,fu)\text{,}
\end{equation*}%
\begin{equation*}
2\rho -d(fu,u_{0})\leq \rho +Ld(u,fu)
\end{equation*}%
and
\begin{equation}
\rho \leq d(fu,u_{0})+Ld(u,fu)\leq \rho +Ld(u,fu)\text{.}  \label{eqn23}
\end{equation}%
Similarly to the arguments used in the proof of Theorem \ref{thm21}, a
direct computation shows that the circle $C_{u_{0},\rho }$ is fixed by $f$.
\end{proof}

\begin{remark}
Notice that, if we consider the case $L=-1$ in the condition $(1)$ of
Theorem \ref{thm22} for $u\in C_{u_{0},\rho }$ then we get%
\begin{equation*}
d(u,fu)\leq \theta _{\rho }(d(u,u_{0}))+d(fu,u_{0})-2d(u,u_{0})=\rho
+d(fu,u_{0})-2\rho =\varphi (u)+\varphi (fu)-2\rho \text{.}
\end{equation*}%
Hence the condition $(C1)^{\ast }$ is satisfied. Also, the condition $(2)$
of Theorem \ref{thm22} is contained in the condition $(C2)^{\ast }$.
Therefore, Theorem \ref{thm22} is a special case of Theorem \ref{thm12} in
this case. For the cases $L\in \left( -\infty ,0\right) $, Theorem \ref%
{thm22} is a generalization of Theorem \ref{thm12}.
\end{remark}

Now, we give some illustrative examples.

\begin{example}
\label{exm:24} Consider the usual metric space $\left(
\mathbb{R}
,d\right) $ and the circle $C_{0,1}=\left\{ -1,1\right\} $. Define the map $%
f_{4}:%
\mathbb{R}
\rightarrow
\mathbb{R}
$ as%
\begin{equation*}
f_{4}x=\left\{
\begin{array}{ccc}
\frac{1}{x} & ; & u\in \left\{ -1,1\right\} \\
2x & ; & \text{otherwise}%
\end{array}%
\right. \text{,}
\end{equation*}%
for each $x\in
\mathbb{R}
$, hence $f_{4}$ satisfies the hypothesis of Theorem \ref{thm22} for $L=-%
\frac{1}{2}$. Clearly, $C_{0,1}$ is the fixed circle of $f_{4}$. It is easy
to check that $f_{4}$ does not satisfy the condition $(1)$ of Theorem \ref%
{thm21} for any $L\in \left( -\infty ,0\right] $.
\end{example}

\begin{example}
\label{exm:25} Consider the usual metric space $\left(
\mathbb{R}
,d\right) $ and the circles $C_{0,1}=\left\{ -1,1\right\} $ and $%
C_{1,2}=\left\{ -1,3\right\} $. Define the self-mapping $f_{5}:%
\mathbb{R}
\rightarrow
\mathbb{R}
$ as%
\begin{equation*}
f_{5}x=\left\{
\begin{array}{ccc}
x & ; & x\in C_{0,1}\cup C_{1,2} \\
\alpha x & ; & \text{otherwise}%
\end{array}%
\right. \text{,}
\end{equation*}%
for each $x\in
\mathbb{R}
$ and $\alpha \geq 2$, then $f_{5}$ satisfies the hypothesis of Theorem \ref%
{thm22} for $L=0$ and for each of the circles $C_{0,1}$ and $C_{1,2}$.
Clearly, $C_{0,1}$ and $C_{1,2}$ are the fixed circles of $f_{5}$. Notice
that the fixed circles $C_{0,1}$ and $C_{1,2}$ are not disjoint.
\end{example}

Considering Example \ref{exm:23} and Example \ref{exm:25}, we deduce that a
fixed circle need not to be unique in Theorem \ref{thm21} and Theorem \ref%
{thm22}. If a fixed circle is non unique then two fixed circle of a
self-mapping can be disjoint or not. Next, we prove a theorem where $f$
fixes a unique circle.

\begin{theorem}
\label{thm23} Let $(X,d)$ be a metric space and $f:X\rightarrow X$ be a
self-mapping which fixes the circle $C_{u_{0},\rho }$. If the condition%
\begin{equation}
d(fu,fv)<\max \left\{ d(v,fu),d(v,fv)\right\} \text{,}  \label{eqn24}
\end{equation}%
is satisfied by $f$ for all $u\in C_{u_{0},\rho }$ and $v\in X\setminus
C_{u_{0},\rho }$, then $C_{u_{0},\rho }$ is the unique fixed circle of $f$.
\end{theorem}

\begin{proof}
Let $C_{u_{1},\mu }$ be another fixed circle of $f$. If we take $u\in
C_{u_{0},\rho }$ and $v\in C_{u_{1},\mu }$ with $u\neq v$, then using the
inequality (\ref{eqn24}), we obtain%
\begin{eqnarray*}
d(u,v) &=&d(fu,fv) \\
&<&\max \left\{ d(v,fu),d(v,fv)\right\} =d(u,v)\text{,}
\end{eqnarray*}%
a contradiction. We have $u=v$ for all $u\in C_{u_{0},\rho }$, $v\in
C_{u_{1},\mu }$ hence $f$ only fixes the circle $C_{u_{0},\rho }.$
\end{proof}

In the following example, we show that the converse of Theorem \ref{thm23}
is not true in general.

\begin{example}
\label{exm:26} Consider the usual metric space $\left(
\mathbb{C}
,d\right) $ and the circle $C_{0,\frac{1}{4}}.$ Define $f_{6}$ on $\mathbb{C}
$ as follows$:$

\begin{equation*}
f_{6}z=\left\{
\begin{array}{ccc}
\frac{1}{16\overline{z}} & \text{if} & z\neq 0 \\
0 & \text{if} & z=0%
\end{array}%
\right. \text{,}
\end{equation*}%
for $z\in
\mathbb{C}
$, where $\overline{z}$ denotes the complex conjugate of $z$. It is not
difficult to see that $C_{0,\frac{1}{4}}$ is the unique fixed circle of $%
f_{6}$ where $f_{6}$ does not satisfy the hypothesis of Theorem \ref{thm23}.
\end{example}

Now, we give the following example as an illustration of Theorem \ref{thm23}.

\begin{example}
\label{exm:27} Let $Y=\left\{ -1,0,1\right\} $ and the metric $d:Y\times
Y\rightarrow \left[ 0,\infty \right) $ be defined by%
\begin{equation*}
d(u,v)=\left\{
\begin{array}{ccc}
0 & ; & u=v \\
\left\vert u\right\vert +\left\vert v\right\vert & ; & u\neq v%
\end{array}%
\right. \text{,}
\end{equation*}%
for all $u\in Y$. If we consider the self-mapping $f_{7}:Y\rightarrow Y$
defined by%
\begin{equation*}
f_{7}u=0\text{,}
\end{equation*}%
for any $u\in Y$, then $C_{1,1}=\left\{ 0\right\} $ is the unique fixed
circle of $f_{7}.$
\end{example}

Next, we present the following interesting theorem that involves the
identity map $I_{X}:X\rightarrow X$ defined by $I_{X}(u)=u$ for all $u\in X.$

\begin{theorem}
\label{thm24} Let $(X,d)$ be a metric space. Consider the map $f$ from $X$
to itself with the fixed circle $C_{u_{0},\rho }$. The self-mapping $f$
satisfies the condition%
\begin{equation}
d(u,fu)\leq \alpha \left[ \max \left\{ d(u,fu),d(u_{0},fu)\right\}
-d(u_{0},fu)\right] \text{,}  \label{identity_con_1}
\end{equation}%
for all $u\in X$ and some $\alpha \in \left( 0,1\right) $ if and only if $%
f=I_{X}$.
\end{theorem}

\begin{proof}
Let $u\in X$ with $fu\neq u$. By inequality (\ref{identity_con_1}), if $%
d(u,fu)\geq d(u_{0},fu)$, then we get
\begin{equation*}
d(u,fu)\leq \alpha \left[ d(u,fu)-d(u_{0},fu)\right] \leq \alpha d(u,fu)%
\text{,}
\end{equation*}%
which leads us to a contradiction due to the fact that $\alpha \in \left(
0,1\right) $. If $d(u,fu)\leq d(u_{0},fu)$, then we find%
\begin{equation*}
d(u,fu)\leq \alpha \left[ d(u_{0},fu)-d(u_{0},fu)\right] =0.
\end{equation*}

Hence, $fu=u$ and that is $f=I_{X}$ since $u$ is an arbitrary in $X$.

Conversely, $I_{X}$ satisfies the condition (\ref{identity_con_1}) clearly.
\end{proof}

\begin{corollary}
Let $(X,d)$ be a metric space and $f:X\rightarrow X$ be a self-mapping. If $%
f $ satisfies the hypothesis of Theorem \ref{thm21} $($resp. Theorem \ref%
{thm22}$)$ but the condition $($\ref{identity_con_1}$)$ is not satisfied,
then $f\neq I_{X}$.
\end{corollary}

Now, we rewrite the following theorem given in \cite{Ozgur-malaysian}.

\begin{theorem}
\cite{Ozgur-malaysian} \label{thm25} Let $(X,d)$ be a metric space. Consider
the map $f$ from $X$ to itself which have a fixed circle $C_{u_{0},\rho }$
and $\varphi $ as in $($\ref{funtion_phi}$)$. Then $f$ satisfies the
condition%
\begin{equation}
d(u,fu)\leq \frac{\varphi (u)-\varphi (fu)}{h}\text{,}  \label{identity2}
\end{equation}%
for every $u\in Y$ and $h>1$ if and only if $f=I_{X}$.
\end{theorem}

\begin{theorem}
\label{thm26} Let $(X,d)$ be a metric space. Consider the map $f$ from $X$
to itself which have a fixed circle $C_{u_{0},\rho }$ and $\varphi $ as in $%
( $\ref{funtion_phi}$)$. Then $f$ satisfies $($\ref{identity_con_1}$)$ if
and only if $f$ satisfies $($\ref{identity2}$)$.
\end{theorem}

\begin{proof}
The proof follows easily.
\end{proof}

\section{$F_{c}$-\textbf{contractive and }$F_{c}$-\textbf{expanding mappings
in metric spaces}}

\label{sec:2}

In this section, we use a different approach to obtain new fixed-circle
results. First, we recall the definition of the following family of
functions which was introduced by Wardowski in \cite{Wardowski}.

\begin{definition}
\cite{Wardowski} \label{def31} Let $\mathbb{F}$ be the family of all
functions $F:(0,\infty )\rightarrow
\mathbb{R}
$ such that

$(F_{1})$ $F$ is strictly increasing,

$(F_{2})$ For each sequence $\left\{ \alpha _{n}\right\} $ in $\left(
0,\infty \right) $ the following holds%
\begin{equation*}
\underset{n\rightarrow \infty }{\lim }\alpha _{n}=0\text{ if and only if }%
\underset{n\rightarrow \infty }{\lim }F(\alpha _{n})=-\infty \text{,}
\end{equation*}

$(F_{3})$ There exists $k\in (0,1)$ such that $\underset{\alpha \rightarrow
0^{+}}{\lim }\alpha ^{k}F(\alpha )=0$.
\end{definition}

Some examples of functions that satisfies the conditions $(F_{1})$, $(F_{2})$
and $(F_{3})$ of Definition \ref{def31} are $F(u)=\ln (u)$, $F(u)=\ln (u)+u$%
, $F(u)=-\frac{1}{\sqrt{u}}$ and $F(u)=\ln (u^{2}+u)$ (see \cite{Wardowski}
for more details).

At this point, we introduce the following new contraction type.

\begin{definition}
\label{def32} Let $(X,d)$ be a metric space and $f$ be a self-mapping on $X$%
. If there exist $t>0$, $F\in \mathbb{F}$ and $u_{0}\in X$ such that
\begin{equation*}
d(u,fu)>0\Rightarrow t+F(d(u,fu))\leq F(d(u_{0},u))\text{,}
\end{equation*}%
for all $u\in X$, then $f$ is called as an $F_{c}$-contraction.
\end{definition}

We note that the point $u_{0}$ mentioned in Definition \ref{def32} must be a
fixed point of the mapping $f$. Indeed, if $u_{0}$ is not a fixed point of $%
f $, then we have $d(u_{0},fu_{0})>0$ and hence%
\begin{equation*}
d(u_{0},fu_{0})>0\Rightarrow t+F(d(u_{0},fu_{0}))\leq F(d(u_{0},u_{0}))\text{%
.}
\end{equation*}%
This is a contradiction since the domain of $F$ is $(0,\infty )$.
Consequently, we obtain the following proposition as an immediate
consequence of Definition \ref{def32}.

\begin{proposition}
\label{prop31} Let $(X,d)$ be a metric space. If $f$ is an $F_{c}$%
-contraction with $u_{0}\in $ $X$ then we have $fu_{0}=u_{0}.$
\end{proposition}

Using this new type contraction we give the following fixed-circle theorem.

\begin{theorem}
\label{thm31} Let $(X,d)$ be a metric space and $f$ be an $F_{c}$%
-contraction with $u_{0}\in $ $X$. Define the number $\sigma $ by
\begin{equation*}
\sigma =\inf \left\{ d(u,fu):u\neq fu,u\in X\right\} \text{.}
\end{equation*}%
Then $C_{u_{0},\sigma }$ is a fixed circle of $f$. In particular, $f$ fixes
every circle $C_{u_{0},r}$ where $r<\sigma $.
\end{theorem}

\begin{proof}
If $\sigma =0$ then clearly $C_{u_{0},\sigma }=\left\{ u_{0}\right\} $ and
by Proposition \ref{prop31}, we see that $C_{u_{0},\sigma }$ is a fixed
circle of $f$. Assume $\sigma >0$ and let $u\in C_{u_{0},\sigma }$. If $%
fu\neq u$, then by the definition of $\sigma $ we have $d(u,fu)\geq \sigma $%
. Hence using the $F_{c}$-contractive property and the fact that $F$ is
increasing, we obtain%
\begin{equation*}
F(\sigma )\leq F(d(u,fu))\leq F(d(u_{0},u))-t<F(d(u_{0},u))=F(\sigma )\text{,%
}
\end{equation*}%
which leads to a contradiction. Therefore, we have $d(u,fu)=0$, that is, $%
fu=u$. Consequently, $C_{u_{0},\sigma }$ is a fixed circle of $f$.

Now we show that $f$ also fixes any circle $C_{u_{0},r}$ with $r<\sigma $.
Let $u\in C_{u_{0},r}$ and assume that $d(u,fu)>0$. By the $F_{c}$%
-contractive property, we have%
\begin{equation*}
F(d(u,fu))\leq F(d(u_{0},u))-t<F(r)\text{.}
\end{equation*}%
Since $F$ is increasing, then we find%
\begin{equation*}
d(u,fu)<r<\sigma .
\end{equation*}%
But $\sigma =\inf \left\{ d(u,fu):\text{for all }u\neq fu\right\} $, which
leads us to a contradiction. Thus, $d(u,fu)=0$ and $fu=u$. Hence, $%
C_{u_{0},r}$ is a fixed circle of $f$.
\end{proof}

\begin{remark}
\label{rem31} $1)$ Notice that, in Theorem \ref{thm31}, the $F_{c}$%
-contraction $f$ fixes the disc $D_{u_{0},\sigma }$. Therefore, the center
of any fixed circle is also fixed by $f$. In Theorem \ref{aip}, the
self-mapping $f$ maps $C_{u_{0},\rho }$ into $($or onto$)$ itself, but the
center of the fixed circle need not to be fixed by $f$.

$2)$ Related to the number of the elements of the set $X$, the number of the
fixed circles of an $F_{c}$-contractive self-mapping $f$ can be infinite $($%
see Example \ref{exm:33}$)$.
\end{remark}

We give some illustrative examples.

\begin{example}
\label{exm:31} Let $X=\left\{ 0,1,e^{2},-e^{2},e^{2}-1,e^{2}+1\right\} $ be
the metric space with the usual metric. Define the self-mapping $%
f_{8}:X\rightarrow X$ as%
\begin{equation*}
f_{8}u=\left\{
\begin{array}{ccc}
1 & ; & u=0 \\
u & ; & \text{otherwise}%
\end{array}%
\right. \text{,}
\end{equation*}%
for all $u\in X$. Then the self-mapping $f_{8}$ is an $F_{c}$-contractive
self-mapping with $F=\ln u$, $t=1$ and $u_{0}=e^{2}$. Using Theorem \ref%
{thm31}, we obtain $\sigma =1$ and $f_{8}$ fixes the circle $%
C_{e^{2},1}=\left\{ e^{2}-1,e^{2}+1\right\} $. Clearly, $\mathbb{C}_{8}$
fixes the disc $D_{e^{2},1}=\left\{ u\in Y:d(u,e^{2})\leq 1\right\} =\left\{
e^{2},e^{2}-1,e^{2}+1\right\} $. Notice that $f_{8}$ fixes also the circle $%
C_{0,e^{2}}=\left\{ -e^{2},e^{2}\right\} .$
\end{example}

The converse statement of Theorem \ref{thm31} is not always true as seen in
the following example.

\begin{example}
\label{exm:32} Let $(X,d)$ be a metric space, $u_{0}\in X$ any point and the
self-mapping $f_{9}:X\rightarrow X$ defined as%
\begin{equation*}
f_{9}u=\left\{
\begin{array}{ccc}
u & ; & d(u,u_{0})\leq \mu \\
u_{0} & ; & d(u,u_{0})>\mu%
\end{array}%
\right. \text{,}
\end{equation*}%
for all $u\in X$ with any $\mu >0$. Then it can be easily seen that $f_{9}$
is not an $F_{c}$-contractive self-mapping for the point $u_{0}$ but $f_{9}$
fixes every circle $C_{u_{0},r}$ where $r\leq \mu $.
\end{example}

\begin{example}
\label{exm:33} Let $\left(
\mathbb{C}
,d\right) $ be the usual metric space and define the self-mapping $f_{10}:%
\mathbb{C}
\rightarrow
\mathbb{C}
$ as%
\begin{equation*}
f_{10}u=\left\{
\begin{array}{ccc}
u & ; & \left\vert u\right\vert <2 \\
u+1 & ; & \left\vert u\right\vert \geq 2%
\end{array}%
\right. \text{,}
\end{equation*}%
for all $u\in
\mathbb{C}
$. We have $\sigma =\min \left\{ d(u,f_{10}u):u\neq f_{10}u\right\} =1$.
Then $f_{10}$ is an $F_{c}$-contractive self-mapping with $F=\ln u$, $t=\ln
2 $ and $u_{0}=0\in
\mathbb{C}
$. Evidently, the number of the fixed circles of $f_{10}$ is infinite.
\end{example}

Now, to obtain a new fixed-circle theorem, we use the well-known fact that
if a self-mapping $f$ on $X$ is surjective, then there exists a self mapping
$f^{\ast }:X\rightarrow X$ such that the map $(f\circ f^{\ast })$ is the
identity map on $X$.

\begin{definition}
\label{def33} A self-mapping $f$ on a metric space $X$ is called as an $%
F_{c} $-expanding map if there exist $t<0$, $F\in \mathbb{F}$ and $u_{0}\in
X $ such that
\begin{equation*}
d(u,fu)>0\Rightarrow F(d(u,fu))\leq F(d(u_{0},fu))+t\text{,}
\end{equation*}%
for all $u\in X$.
\end{definition}

\begin{theorem}
\label{thm32} Let $(X,d)$ be a metric space. If $f:X\rightarrow X$ is a
surjective $F_{c}$-expanding map with $u_{0}\in X$, then $f$ has a fixed
circle in $X.$
\end{theorem}

\begin{proof}
Since $f$ is surjective, we know that there exists a self-mapping $f^{\ast
}:X\rightarrow X,$ such that the map $(f\circ f^{\ast })$ is the identity
map on $X$. Let $u\in X$ be such that $d(u,f^{\ast }u)>0$ and $z=f^{\ast }u$%
. First, notice the following fact
\begin{equation*}
fz=f(f^{\ast }u)=(f\circ f^{\ast })u=u\text{.}
\end{equation*}%
Since
\begin{equation*}
d(z,fz)=d(fz,z)>0\text{,}
\end{equation*}

now, by applying the $F_{c}$-expanding property of $f$ we get%
\begin{equation*}
F\left( d(z,fz)\right) \leq F(d(u_{0},fz))+t
\end{equation*}%
and
\begin{equation*}
F\left( d(f^{\ast }u,u)\right) \leq F(d(u_{0},u))+t\text{.}
\end{equation*}%
Therefore, we obtain
\begin{equation*}
-t+F\left( d(f^{\ast }u,u)\right) \leq F(d(u_{0},u))\text{.}
\end{equation*}%
Consequently, $f^{\ast }$ is an $F_{c}$-contraction on $X$ with $u_{0}$ as $%
-t>0$. Then by Theorem \ref{thm31}, $f^{\ast }$ has a fixed circle $%
C_{u_{0},\sigma }$. Let $v\in C_{u_{0},\sigma }$ be any point. Using the
fact that
\begin{equation*}
fv=f(f^{\ast }v)=v\text{,}
\end{equation*}%
we deduce that $fv=v$, that is $v$ is a fixed point of $f$, which implies
that $f$ also fixes $C_{u_{0},\sigma }$, as required.
\end{proof}

\begin{example}
Let $X=\left\{ 1,2,3,4,5\right\} $ with the usual metric. Define the
self-mapping $f_{11}:X\rightarrow X$ by%
\begin{equation*}
f_{11}u=\left\{
\begin{array}{ccc}
2 & ; & u=1 \\
1 & ; & u=2 \\
u & ; & u\in \left\{ 3,4,5\right\}%
\end{array}%
\right. \text{.}
\end{equation*}%
$f_{11}$ is a surjective $F_{c}$-expanding map with $u_{0}=4$, $F(u)=\ln u$
and $t=-\ln 2$. We have%
\begin{equation*}
\sigma =\min \left\{ d\left( u,fu\right) :u\neq fu,u\in X\right\} =1
\end{equation*}%
and the circle $C_{4,1}=\left\{ 3,5\right\} $ is the fixed circle of $f$.
\end{example}

\begin{remark}
If $f$ is not a surjective map, then the result in Theorem \ref{thm32} is
not true everywhen. For example, let $X=\left\{ 1,2,3,4\right\} $ with the
usual metric $d.$ Define the self-mapping $f_{12}:X\rightarrow X$ by
\begin{equation*}
f_{12}u=\left\{
\begin{array}{ccc}
2 & ; & u\in \left\{ 1,3\right\} \\
1 & ; & u=2 \\
4 & ; & u=4%
\end{array}%
\right. \text{.}
\end{equation*}%
Then, it is easy to check that $f_{12}$ satisfies the condition
\begin{equation*}
d(u,fu)>0\Rightarrow F(d(u,fu))\leq F(d(u_{0},fu))+t
\end{equation*}%
for all$\ u\in X$, with $F\left( u\right) =\ln u$, $u_{0}=4$ and $t=-\ln 2$.
Therefore, $f_{12}$ satisfies all the conditions of Theorem \ref{thm32},
except that $f_{12}$ is not surjective. Notice that $\sigma =1$ and $f_{12}$
does not fix the circle $C_{4,1}$.
\end{remark}

\section{\textbf{Fixed point sets of activation functions}}

\label{sec:3}

Activation functions are the primary neural networks decision-making units
in a neural network and hence it is critical to choose the most appropriate
activation function for neural network analysis \cite{Szandala}.
Characteristic properties of activation functions play an important role in
learning and stability issues of a neural network. A comprehensive analysis
of different activation functions with individual real-world applications
was given in \cite{Szandala}. We note that the fixed point sets of commonly
used activation functions (e.g. Ramp function, ReLU function, Leaky ReLU
function) contain some fixed discs and fixed circles. For example, let us
consider the Leaky ReLU function defined by%
\begin{equation*}
f(x)=\max (kx,x)=\left\{
\begin{array}{ccc}
kx & ; & x\leq 0 \\
x & ; & x>0%
\end{array}%
\right. \text{,}
\end{equation*}%
where $k\in \left[ 0,1\right] $. In \cite{Xu}, the Leaky-Reluplex algorithm
was proposed to verify Deep Neural Networks (DNNs) with Leaky ReLU
activation function (see \cite{Xu} for more details). Now we consider the
fixed point set of the Leaky ReLU activation function by a geometric
viewpoint. Let $\rho =u_{0}\in \left( 0,\infty \right) $ be any positive
number and consider the circle $C_{u_{0},\rho }=\left\{ 0,2u_{0}\right\} $.
Then it is easy to check that the function $f(x)$ satisfies the conditions
of Theorem \ref{thm21} for the circle $C_{u_{0},\rho }$ with $L=0$. Clearly,
the circle $C_{u_{0},\rho }$ is a fixed circle of $f(x)$ and the center of
the fixed circle is also fixed by $f(x)$.

On the other hand, theoretic fixed point theorems have been extensively used
in the study of neural networks. For example, in \cite{Li}, the existence of
a fixed point for every recurrent neural network was shown and a geometric
approach was used to locate where the fixed points are. Brouwer's Fixed
Point Theorem was used to ensure the existence of a fixed point. This study
shows the importance of the geometric viewpoint and theoretic fixed point
results in applications. Obviously, our fixed circle and fixed disc results
are important for future studies in the study of neural networks.

\section{\textbf{Conclusion and future works}}

\label{sec:4}

In this section, we want to bring to the reader's attention in connection
with the investigation of some open questions. Concerning the geometry of
non unique fixed points of a self-mapping on a metric space, we have
obtained new geometric (fixed-circle or fixed-disc) results. To do this, we
use two different approaches. One of them is to measure whether a given
circle is fixed or not by a self-mapping. Another approach is to find which
circle is fixed by a self-mapping under some contractive or expanding
conditions. The investigation of new conditions which ensure a circle or a
disc to be fixed by a self-mapping can be considered as a future problem.
For a self-mapping of which fixed point set contains a circle or a disc, new
contractive or expanding conditions can also be investigated.

On the other hand, there are some examples of self-mappings which have a
common fixed circle. For example, let $\left(
\mathbb{R}
,d\right) $ be the usual metric space and consider the circle $%
C_{0,1}=\left\{ -1,1\right\} $. We define the self-mappings $f_{13}:%
\mathbb{R}
\rightarrow
\mathbb{R}
$ and $f_{14}:%
\mathbb{R}
\rightarrow
\mathbb{R}
$\ as%
\begin{equation*}
f_{13}x=\left\{
\begin{array}{ccc}
\frac{1}{x} & ; & x\in \left\{ -1,1\right\} \\
0 & ; & \text{otherwise}%
\end{array}%
\right. \text{ and }f_{14}x=\frac{5x+3}{3x+5}\text{,}
\end{equation*}%
for each $x\in
\mathbb{R}
$, respectively. Then both the self-mappings $f_{13}$ and $f_{14}$ fixes the
circle $C_{0,1}=\left\{ -1,1\right\} $, that is, the circle $C_{0,1}=\left\{
-1,1\right\} $ is a common fixed circle of the self-mappings $f_{13}$ and $%
f_{14}$. At this point, the following question can be left as a future study.

\begin{question}
What is (are) the condition(s) to make any circle $C_{u_{0},\rho }$ as the
common fixed circle for two (or more than two) self-mappings?
\end{question}

Finally, the problems considered in this paper can also be studied on some
generalized metric spaces. For example, the notion of an $M_{s}$-metric
space was introduced in \cite{M4}.

\begin{notation}
We use the following notations.

1. $m_{{s}_{u,v,z}}:=min\{m_{s}(u,u,u),m_{s}(v,v,v),m_{s}(z,z,z)\}$

2. $M_{{s}_{u,v,z}}:=max\{m_{s}(u,u,u),m_{s}(v,v,v),m_{s}(z,z,z)\}$
\end{notation}

\begin{definition}
An $M_{s}$-metric on a nonempty set $Y$ is a function $m_{s}:Y^{3}%
\rightarrow
\mathbb{R}
^{+}$ if for all $u,v,z,t\in Y$ we have

\begin{enumerate}
\item $m_{s}(u,u,u)=m_{s}(v,v,v)=m_{s}(z,z,z)=m_{s}(u,v,z)%
\Longleftrightarrow u=v=z,$

\item $m_{{s}_{u,v,z}}\leq m_{s}(u,v,z),$

\item $m_{s}(u,u,v)=m_{s}(v,v,u),$

\item
\begin{align*}
(m_{s}(u,v,z)-m_{{s}_{u,v,z}})\leq & (m_{s}(u,u,t)-m_{{s}_{u,u,t}}) \\
& +(m_{s}(v,v,t)-m_{{s}_{v,v,t}})+(m_{s}(z,z,t)-m_{{s}_{z,z,t}}).
\end{align*}
\end{enumerate}

Then the pair $(Y,m_{s})$ is called an $M_{s}$-metric space.
\end{definition}

One can consult \cite{M4} for some examples and basic notions of an $M_{s}$%
-metric space.

In $M_{s}$-metric spaces we define a circle as follow;
\begin{equation*}
C_{u_{0},\rho }=\{u\in Y\mid m_{s}(u_{0},u,u)-m_{{s}_{{u_{0},u,u}}}=\rho \}.
\end{equation*}

\begin{question}
Let $(Y,m_{s})$ be an $M_{s}$-metric space, $k>1$ and $f$ be a surjective
self-mapping on $Y$. Let we have
\begin{equation*}
m_{s}(u,fu,f^{2}u)\leq km_{s}(u_{0},u,fu),
\end{equation*}%
for every $u\in Y$ and some $u_{0}\in Y$. Does $f$ have point circle on $Y?$
\end{question}

\begin{question}
Let $(Y,m_{s})$ be an $M_{s}$-metric space, $t>0$, $F\in \mathbb{F}$ and $f$
be a surjective self-mapping on $Y$. Let we have
\begin{equation*}
m_{s}(u,fu,f^{2}u)>0\Rightarrow F(m_{s}(u,fu,f^{2}u))\geq
F(m_{s}(u_{0},u,fu))+t,
\end{equation*}%
for every $u\in Y$ and some $u_{0}\in Y$. Does $f$ have a fixed circle on $%
Y? $
\end{question}
\section*{Acknowledgements}
The first author would like to thank Prince Sultan University for funding this work through research group Nonlinear
Analysis Methods in Applied Mathematics (NAMAM) group number RG-DES-2017-01-17.

\end{document}